\documentclass{article}

\usepackage{fullpage}
\usepackage{verbatim}
\usepackage{graphics}
\usepackage[pdftex]{graphicx}
\usepackage{amsmath, amssymb, amsfonts, amsthm}
\usepackage{url}
\usepackage{setspace}
\usepackage{algorithmic}
\usepackage{enumitem}
\setlist{nolistsep}

\def \PG[#1,#2]{PG(#1,#2)}
\def \AG[#1,#2]{\mathbb{F}_{#2}^{#1}}

\newtheorem{theorem}{Theorem}

\newtheorem{lemma}[theorem]{Lemma}

\newtheorem{corollary}[theorem]{Corollary}

\DeclareMathOperator*{\argmax}{arg\,max}

\begin{document}

\title{Incidences and pairs of dot products\footnote{Work on this paper was supported by NSF grant CCF-1350572.}}
\author{Ben Lund  \\}
\maketitle

\begin{abstract}
Let $\mathbb{F}$ be a field, let $P \subseteq \mathbb{F}^d$ be a finite set of points, and let $\alpha,\beta \in \mathbb{F} \setminus \{0\}$.
We study the quantity
\[|\Pi_{\alpha, \beta}| = \{(p,q,r) \in P \times P \times P \mid p \cdot q = \alpha, p \cdot r = \beta \}.\]

We observe a connection between the question of placing an upper bound on $|\Pi_{\alpha,\beta}|$ and a well-studied question on the number of incidences betwen points and hyperplanes, and use this connection to prove new and strengthened upper bounds on $|\Pi_{\alpha,\beta}|$ in a variety of settings.
\end{abstract}

\section{Introduction}\label{sec:intro}

Let $\mathbb{F}$ be a field, let $P \subseteq \mathbb{F}^d$ be a finite set of points, and let $\alpha,\beta \in \mathbb{F} \setminus \{0\}$.
Denote
\[\Pi_{\alpha, \beta} = \Pi_{\alpha, \beta}(P) = \{(p,q,r) \in P \times P \times P \mid p \cdot q = \alpha, p \cdot r = \beta \}.\]

The quantity $\max_{P,\alpha,\beta:|P|=n}(|\Pi_{\alpha,\beta}|)$ was first investigated by Barker and Senger \cite{barker2015upper}, who gave upper bounds on $|\Pi_{\alpha,\beta}|$ in terms of $|P|$ for $P \subset \mathbb{R}^2$.
The case that $P$ is a sufficiently large subset of a vector space over a finite field, or of a module over the set of integers modulo the power of a prime, was investigated by Covert and Senger \cite{covert2015pairs}.

We observe that known upper bounds on the maximum number of incidences between a set of hyperplanes and points imply upper bounds on $|\Pi_{\alpha,\beta}|$.
We use this approach to obtain new upper bounds on the size of $|\Pi_{\alpha, \beta}|$ under various restrictions on $\mathbb{F}$ and $P$.
In the case $P \subset \mathbb{R}^2$, the new bounds strengthen and generalize the results of Barker and Senger.

Our first result shows that $|\Pi_{\alpha,\beta}| \leq O(n^2)$, and gives stronger bounds when no line contains too many points of $P$.

\begin{theorem}\label{thm:generalPlane}
Let $P \subseteq \mathbb{F}^2$ be a set of $n$ points such that no $s$ points of $P$ are collinear, and let $\alpha,\beta \in \mathbb{F} \setminus \{0\}$.
Then,
\[|\Pi_{\alpha,\beta}| < \min(2s^2n, 4n^2).\]
\end{theorem}

It is possible that $|\Pi_{\alpha,\beta}| \geq \Omega(n^2)$.
For example, consider the set of points $P = p \cup P' \subset \mathbb{R}^2$, where $p$ has the coordinates $(1,1)$ and $P'$ is contained in the line $x+y = 1$.
Then, if $q,r \in P'$, we have $(p,q,r) \in \Pi_{1,1}$.
Hence, $|\Pi_{1,1}| \geq |P'|^2$.

Barker and Senger \cite{barker2015upper} showed that $|\Pi_{\alpha,\beta}| \leq O(n^2)$ when $P \subset \mathbb{R}^2$, and Covert and Senger \cite{covert2015pairs} showed that $|\Pi_{\alpha,\beta}| \leq O(n^2)$ when $P \subset \mathbb{F}_q^2$.
The observation that we can obtain a better bound if $P$ contains no large collinear set is new.

In the case that $P \subset \mathbb{R}^2$ or $P \subset \mathbb{C}^2$, we can use the Szemer\'edi-Trotter theorem \cite{szemeredi1983extremal, toth2015szemeredi, zahlszemeredi} to get an improvement to the conclusion of Theorem \ref{thm:generalPlane} for $s\geq \Omega(n^{1/3})$.

\begin{theorem}\label{thm:R2}
Let $P$ be a set of $n$ points in $\mathbb{F}^2$, for either $\mathbb{F} = \mathbb{R}$ or $\mathbb{F} = \mathbb{C}$.
Suppose that no $s$ points of $P$ are collinear, and let $\alpha,\beta \in \mathbb{F} \setminus \{0\}$.
Then, 
\[|\Pi_{\alpha,\beta}| \leq O(n^{5/3} + sn)\]
\end{theorem}

In Section \ref{sec:constructions}, we describe a very simple construction that achieves $|\Pi_{\alpha,\beta}| = \Omega(sn)$ for any $s,n$.
This shows that Theorem \ref{thm:R2} is tight up to constant factors when $s \geq n^{2/3}$.
It is an open problem to determine the asymptotically least possible upper bound on $|\Pi_{\alpha,\beta}|$ for $s \leq n^{2/3}$, under the hypotheses of Theorem \ref{thm:R2}.

It's worth noting that the conclusion of Theorem \ref{thm:R2} does not hold for subsets of $\mathbb{F}_q^2$.
Covert and Senger \cite{covert2015pairs} showed that, for $P \subseteq \mathbb{F}_q^2$ with $n = |P| \geq \Omega_\epsilon(q^{3/2+\epsilon})$ for some $\epsilon > 0$ \footnote{The subscript in the notation $O_\epsilon$ indicates that the constants hidden in the $O$-notation depend on $\epsilon$.}, we have $|\Pi_{\alpha,\beta}| = (1-o(1))(n^3q^{-2})$.
Hence, in this regime, we have $|\Pi_{\alpha,\beta}| > \omega(nq^{3/2}) > \omega(n^{5/3}) > \min(2q^2n, O(n^{5/3} + qn) = O(n^{5/3})$.

As a corollary of Theorem \ref{thm:R2}, we obtain the following improvement to a result of Barker and Senger.

\begin{corollary}\label{thm:density}
Let $P \subset [0,1]^2$ with $|P|=n$, such that the distance between each pair of points in $P$ is at least $\epsilon$.
Then, $|\Pi_{\alpha,\beta}| \leq O(n^{5/3}+n \epsilon^{-1})$.
\end{corollary}

\begin{proof}
We can assume that $\epsilon \leq O(n^{1/2})$, since otherwise $n$ points cannot be placed in $[0,1]^2$ such that the distance between each pair of points is at least $\epsilon$.
Since each pair of points is at distance at least $\epsilon$, and since the maximum distance between any pair of points is at most $\sqrt{2}$, there are at most $O(\epsilon^{-1})$ points on any line.
An application of Theorem \ref{thm:R2} completes the proof.
\end{proof}

Under the hypotheses of Corollary \ref{thm:density}, Barker and Senger showed that $|\Pi_{\alpha,\beta}| \leq O(n^{4/3}\epsilon^{-1}\log(\epsilon^{-1}))$.
In Section \ref{sec:constructions}, we give a construction showing that Corollary \ref{thm:density} is tight up to constant factors when $\epsilon \leq O(n^{1/3})$.

When $P \subset \mathbb{F}_q^2$ for $p$ prime, and $|P|$ is not too large, we obtain a slight improvement to Theorem \ref{thm:generalPlane} when $s$ is close to $n^{1/2}$ by using an incidence bound first proved by Bourgain, Katz, and Tao \cite{bourgain2004sum, jones2012further, roche2014new}.

\begin{theorem}\label{thm:Fp2}
Let $P$ be a set of $n<p$ points in $\mathbb{F}_p^2$, for a prime $p$, such that no $s$ points of $P$ are collinear.
 Then, there exists a constant $\epsilon > 0$ such that
\[|\Pi_{\alpha,\beta}| \leq O( s n^{3/2 - \epsilon}).\]
\end{theorem}

For $P \subset \mathbb{F}^d$ with $d>2$, there is no upper bound of the form $|\Pi_{\alpha,\beta}| \leq o(n^3)$ that holds for an arbitrary set of $n$ points.
For example, let $A \subset \mathbb{R}$ with $|A| = n/2$, and let $P$ be all points with coordinates $(a,0,\beta)$ or $(0,a,1)$, for $a \in A$ and $\beta \in \mathbb{R} \setminus \{0\}$.
If $p \in P$ has the form $(a,0,\beta)$ and $q,r \in P$ have the form $(0,a,1)$, then $(p,q,r) \in \Pi_{\beta,\beta}$.
Since we have $n/2$ choices for $p$, and $(n/2)^2$ choices for $(q,r)$, we have that $|\Pi_{\beta,\beta}| \geq \Omega(n^3)$.

It is possible to obtain a nontrivial upper bound on $|\Pi_{\alpha,\beta}|$ for $P \subset \mathbb{F}^d$ by restricting the maximum possible number of points on a hyperplane.
We obtain a more refined bound by further restricting the number of points on a $(d-2)$-plane\footnote{We refer to a $k$ dimensional affine subspace as a $k$-plane. A hyperplane in $\mathbb{F}^d$ is a $(d-1)$-plane.}.

\begin{theorem}\label{thm:generalHighDim}
Let $P$ be a set of $n$ points in $\mathbb{F}^d$, such that no $s$ points of $P$ are contained in any single hyperplane, and such that no $t$ points of $P$ are contained in any single $(d-2)$-plane.
Let $\alpha,\beta \in \mathbb{F} \setminus \{0\}$.
Then,
\[|\Pi_{\alpha,\beta}| \leq \min(2s^2n,  O(tn^2)).\]
\end{theorem}

Up to constant factors, Theorem \ref{thm:generalHighDim} is a generalization of Theorem \ref{thm:generalPlane}; since the points of $P$ are distinct, $t=1$ in $\mathbb{F}^2$.

Rudnev \cite{rudnev2014number}  proved an upper bound on the number of incidences between points and planes in $\mathbb{F}^3$ that holds for an arbitrary field $\mathbb{F}$ with characteristic other than $2$.
In the case of positive characteristic, application of Rudnev's bound requires that there are not too many points.
Applying Rudnev's bound in the framework of this paper gives

\begin{theorem}\label{thm:F3Rudnev}
Let $P$ be a set of $n$ points in $\mathbb{F}^3$, such that no $s$ points of $P$ are coplanar, and no $t$ points of $P$ are collinear.
If $\mathbb{F}^3$ has positive characteristic $p$, then $p \neq 2$ and $n=O(p^2)$.
Then,
\[|\Pi_{\alpha,\beta}| \leq  O(n^2 \log (sn^{-1/2}) + stn).\]
\end{theorem}

Theorem \ref{thm:F3Rudnev} gives a stronger conclusion than Theorem \ref{thm:generalHighDim} for $s\geq \Omega(n^{1/2})$.

The polynomial partitioning technique of Guth and Katz \cite{guth2015erdHos} has recently led to a number of higher-dimensional incidence bounds in $\mathbb{R}^d$.
Applying one such bound \cite{lund2014bisector}, we find

\begin{theorem}\label{thm:Rd}
Let $P$ be a set of $n$ points in $\mathbb{R}^d$.
Suppose no more than $s$ points lie on any $(d-1)$-plane, and no more than $t$ points lie on any $(d-2)$-plane.
Then, for any $\epsilon > 0$,
\[|\Pi_{\alpha,\beta}| \leq O_{\epsilon,d}(nt^2 + n^{(4d-3)/(2d-1)+\epsilon}t^{(2d-2)/(2d-1)+ \epsilon} + sn).\]
\end{theorem}

Theorem \ref{thm:Rd} falls short of being a generalization of Theorem \ref{thm:R2} in two respects.
First, it is only proved in real space, not complex space.
Second, it is weaker by an arbitrarily small polynomial factor $(nt)^\epsilon$.
Both of these limitations are inherited from the incidence bound used, and could conceivably be removed by future developments in the science of proving incidence bounds.

\section{Constructions}\label{sec:constructions}

In this section, we describe two infinite families of sets of points in $\mathbb{R}^2$.
The first familiy is relatively simple, and shows that a set of $n$ points such that no line contains $s$ of the points can have $|\Pi_{\alpha,\beta}| \geq ns$.
This shows that Theorem \ref{thm:R2} is tight for $s \geq n^{2/3}$.

The second construction shows that that Corollary \ref{thm:density} is tight.
This also implies that Theorem \ref{thm:R2} is tight for $s \geq n^{2/3}$, but says nothing for $s < n^{1/2}$.
The second construction is also slightly more complicated than the first, and results in fewer pairs of dot products by a constant factor.

In both this section and Section \ref{sec:proofs}, we will need the following observation.

\begin{lemma}\label{thm:distinctLines}
Let $p,q \in \mathbb{F}^d$ and $\alpha \in \mathbb{F} \setminus \{0\}$.
The set of points $\ell_p = \{r \mid r \cdot p = \alpha\}$ is a hyperplane.
If $p,q$ are distinct, then $\ell_p, \ell_q$ are distinct.
\end{lemma}

\begin{proof}
Write $p = (p_1, \ldots, p_d), r = (r_1, \ldots, r_d), q = (q_1, \ldots, q_d)$.
The set of points $\ell_p$ satisfies the linear equation
\[p_1 r_1 + \ldots + p_d r_d = \alpha,\]
and so is a hyperplane.
Suppose $\ell_p = \ell_q$.
Then there is some $\beta$ such that $\beta(p_1, \ldots, p_d, \alpha) = (q_1, \ldots, q_d, \alpha)$.
Since $\alpha \neq 0$, we must have $\beta = 1$, so $p = q$.
\end{proof}

\subsection{Simple construction for Theorem \ref{thm:R2}}

For given positive integers $n,s$ such that $n/s$ is an integer, we construct a set $P \subset \mathbb{R}^2$ of $n$ points, no $s$ collinear, such that $|\Pi_{1,1}| \geq n(s-1)^2/s$.

Let $Q = \{q_1, q_2, \ldots, q_{n/s}\}$ be a set of points such that no three are on a line, and such that $q_i \neq (0,0)$ for all $i$.
For $i \in [1,n/s]$, let $\ell_i = \{p \mid p \cdot q_i = 1\}$.
Let $R_i$ be a set of $s-1$ points such that, if $r \in R_i$, then $r \in \ell_i$, and no $s$ points of $P = Q \cup R_1 \cup R_2 \cup \ldots \cup R_{n/s}$ are collinear.
Note that $|P| = n$.
For each pair $(r_{ij},r_{ik}) \in R_i$, we have $(q_i, r_{ij}, r_{ik}) \in \Pi_{1,1}$.
Since there are $n/s$ choices for $q_i$ and $(s-1)^2$ choices for $(r_{ij},r_{ik})$, we have that $|\Pi_{1,1}| \geq n(s-1)^2/s$.

\subsection{Construction for Corollary \ref{thm:density}}

For given $n$ and $\epsilon < n^{-1/2}/3$ satisfying certain divisibility conditions, we construct a set $P \subset [0,1]^2$ of $n + 3n\epsilon = (1 + o(1))n$ points such that the distance between each pair of points in $P$ is at least $\epsilon$, such that $|\Pi_{1/2,1/2}| \geq \Omega(n \epsilon^{-1})$.

Let $L = \{\ell_1, \ell_2, \ldots, \ell_{3n\epsilon}\}$ be a set of lines such that $\ell_j$ contains the points $(0,1)$ and $(1, 1-3\epsilon j)$.
Note that, since $\epsilon < n^{-1/2}/3$, each line of $L$ has positive $y$-coordinate for all $x \in [0,1]$.
Let $Q_i$ for $i \in [1,3n\epsilon]$ be a set $\epsilon^{-1}/3$ points such that each point of $Q_i$ is incident to $\ell_i$, has $x$-coordinate in the interval $[2/3,1]$, and the difference between the $x$-coordinate of each pair of points in $Q_i$ is at least $\epsilon$.
Note that the distance between $\ell_i$ and $\ell_{i+1}$ at $x=2/3$ is $(2/3)3 \epsilon > \epsilon$, so the distance bewteen points in $Q_i$ and $Q_{i+1}$ is at least $\epsilon$.
Let $R = \{r_1, r_2, \ldots, r_{3n\epsilon}\}$ be the set of points such that the coordinates of $r_j$ are $(3 \epsilon j / 2, 1/2)$.
Note that the distance between each pair of points in $R$ is at least $3 \epsilon /2$, and all points of $R$ have $x$-coordinate at most $9n\epsilon^2/2 < 1/2$.
Hence, $P = Q_1 \cup Q_2 \cup \ldots \cup Q_{3n\epsilon} \cup R$ is a set of $(1 + o(1))n$ points such that the distance between each pair of points in $P$ is at least $\epsilon$.

Let $q \in Q_j$ with coordinates $(\lambda, 1-3\lambda\epsilon j)$ for some $\lambda \in [2/3,1]$.
Then $r_j \cdot q = \lambda 3 \epsilon j / 2 + (1 - \lambda 3 \epsilon j) / 2 = 1/2$.
Hence, for $q_{j1}, q_{j2} \in Q_i$, we have $(r_j, q_{j1}, q_{j2}) \in \Pi_{1/2,1/2}$.
Since there are $3n\epsilon$ choices for $r_j$ and $\epsilon^{-2}/9$ choices for $(q_{j1},q_{j2})$, we have $|\Pi_{1/2,1/2}| \geq n \epsilon^{-1}/3$.

\section{Proofs}\label{sec:proofs}

In this section, we prove the main theorems stated in Section \ref{sec:intro}.

The proofs all have the same basic outline.
First, we use a unified reduction from the question studied here to an incidence problem; this reduction is in Section \ref{sec:preliminaries}.
Section \ref{sec:preliminaries} also introduces notation that is used in the subsequent proofs.
Then, we apply known incidence bounds to obtain the concrete results listed in Section \ref{sec:intro}.
These six proofs are organized into two sections; Section \ref{sec:2dim} includes proofs of those bounds that are proved for point sets in a plane, and Section \ref{sec:ddim} has the proofs for bounds in higher dimensions.

\subsection{From Pairs of Dot Products to Incidences}\label{sec:preliminaries}

Suppose we are given a finite point set $P \subset \mathbb{F}^d$ such that no hyperplane contains $s$ points of $P$, and constants $\alpha,\beta \in \mathbb{F} \setminus \{0\}$.

For any hyperplane $h$, denote
\[wt(h) = |h \cap P|.\]

For any point $p \in P$ and constant $c \in \mathbb{F}$, denote
\begin{align*}
h_c(p) &= \{x \in \mathbb{F}^d \mid p \cdot x = c\}, \\
\pi(p) &= \{(q,r) \in P \times P \mid p \cdot q = \alpha, p \cdot r = \beta \}.
\end{align*}

Note that $|\Pi_{\alpha,\beta}| = \sum_{p \in P} |\pi(p)|$, and $|\pi(p)| = wt(h_\alpha(p)) \cdot wt(h_\beta(p)) < wt(h_\alpha(p))^2 + wt(h_\beta(p))^2$.
Hence,
\[|\Pi_{\alpha,\beta}| < \sum_{p \in P} wt(h_\alpha(p))^2 +  \sum_{p \in P} wt(h_\beta(p))^2.\]

Let $\gamma = \argmax_{\gamma \in \{\alpha,\beta\}} \sum_{p \in P} wt(h_\gamma(p))^2$; we have
\[|\Pi_{\alpha,\beta}| < 2 \sum_{p \in P} wt(h_\gamma(p))^2.\]

Let
\[H = \{h_\gamma(p) \mid p \in P\}.\]
Since $\gamma \neq 0$, if $p \neq p'$ then $h_\gamma(p) \neq h_\gamma(p')$, so $|H| = n$.

Denote
\begin{align*}
f_{k} & = |\{h \in H \mid wt(h) \geq k\}|, \\
f_{=k} &= |\{h \in H \mid wt(h) = k\}|.
\end{align*}

Collecting hyperplanes of equal weight, we have
\[\sum_{p \in P} wt(h_\gamma(p))^2 = \sum_{k < s} f_{=k}k^2.\]

We have now established
\begin{lemma}\label{thm:incidence}
Let $P \subset \mathbb{F}^d$ be a finite set of points such that no hyperplane contains more than $s$ points of $P$, and let $\alpha,\beta \in \mathbb{F}\setminus \{0\}$.
Then,
\[|\Pi_{\alpha,\beta}| < 2 \sum_{k < s} f_{=k}k^2.\]
\end{lemma}

Since $|H|=n$, we have as an immediate corollary to Lemma \ref{thm:incidence}

\begin{corollary}\label{thm:s2n}
Under the hypotheses of Lemma \ref{thm:incidence}, we have
\[|\Pi_{\alpha,\beta}| < 2 s^2 n.\]
\end{corollary}

Let $g_{k}$ be a monotonically decreasing function of $k$.
We claim that, if $f_{k} \leq g_{k}$ for all $k$, then 
$\sum_{k < s} k^2 f_{= k} = \sum_{k < s} k^2 (f_{k} - f_{k+1}) \leq \sum_{k < s}k^2 (g_k - g_{k+1})$.
The proof is by induction on the minimum index $j$ such that $f_{i} = g_i$ for all $i > j$.
In the base case, $f_{k} = g_k$ for all $k$.
If $f_{j} \neq g_j$, then $f_j < g_j$.
The function $f'$ such that $f'_k = f_k$ for $k \neq j$ and $f'_j = g_j$ has the property that $\sum_{k \leq s} k^2 (f_k - f_{k+1}) \leq \sum_{k \leq s}k^2(f'_k - f'_{k+1})$, and by induction $\sum_{k \leq s} k^2 (f'_k - f'_{k+1}) \leq \sum_{k \leq s}k^2 (g_k - g_{k+1})$.

Hence, we have

\begin{lemma}\label{thm:k-rich}
Let $P \subset \mathbb{F}^d$ be a finite set of points such that no hyperplane contains more than $s$ points of $P$, and let $\alpha, \beta \in \mathbb{F} \setminus \{0\}$.
Let $g_k$ be a monotonically decreasing function of $k$ such that $g_k \geq f_{k}$ for all $k$.
Then
\[|\Pi_{\alpha,\beta}| \leq 2 \sum_{k \leq s} k^2 \left (g_k - g_{k+1} \right ).\]
\end{lemma}

In the proofs, we will use the following specific function $g_k$ with Lemma \ref{thm:k-rich}, which is related to the maximum number of hyperplanes that can contain at least $k$ of a set of $n$ points.

Let $\mathcal{H}$ be the set of all hyperplanes in $\mathbb{F}^d$, and denote
\begin{align*}
g'_k &= |\{h \in \mathcal{H} : wt(h) \geq k\}|, \\
g_k &=
\begin{cases}
g'_k, & \text{if } g'_k \leq n, \\
n, & \text{otherwise}.
\end{cases}
\end{align*}
Since each hyperplane that contributes to $f_k$ also contributes to $g_k$, we have that $g_k \geq f_k$.
In addition, $g_k$ is monotonically decreasing, and so satisfies the hypotheses of Lemma \ref{thm:k-rich}.

We denote 
\[g_{=k} = g_{k} - g_{k+1};\] in particular, this implies that $\sum_{k} g_{=k} \leq n$.

In the following proofs, we will often derive a bound on $g_k$ from some known bound on the quantity
\[I(P,H) = |\{(p,h) \in P \times H \mid p \in h\}|,\]
in which $P$ and $H$ may be taken to be arbitrary sets of points and hyperplanes, respectively.

\subsection{On a plane}\label{sec:2dim}

In this section, we prove Theorems \ref{thm:generalPlane}, \ref{thm:R2}, and \ref{thm:Fp2}.
These theorems are all for planar point sets, and are united by a common hypothesis (no $s$ points on any line).

\subsubsection{Theorem \ref{thm:generalPlane}}

The proof for arbitrary field $\mathbb{F}$ uses only the facts that each distinct pair of points lies on one line, each distinct pair of lines intersects in at most one point, and $g_k \leq n$ for all $k$.

\begin{proof}[Proof of Theorem \ref{thm:generalPlane}]
There are $k^2$ ordered pairs of (not necessarily distinct) points of $P$ on a line containing $k$ points of $P$, and $n^2$ such pairs in total.
Each distinct pair of points appears on one line, and each line crosses at a single point, so
\[\sum_{k\geq 2} k^2 g_{=k} \leq n^2 + \binom{g_2}{2} < 2n^2.\]
Combined with Lemma \ref{thm:incidence} and Corollary \ref{thm:s2n}, this completes the proof.
\end{proof}

\subsubsection{Theorem \ref{thm:R2}}

For $\mathbb{F} = \mathbb{R}$ or $\mathbb{C}$, we can use the Szemer\'edi-Trotter theorem, proved for $\mathbb{R}$ by Szemer\'edi and Trotter \cite{szemeredi1983extremal}.
Since the same bound was proved for $\mathbb{C}$ by T\'oth \cite{toth2015szemeredi}, and later by completely different methods by Zahl \cite{zahlszemeredi}, we have a unified proof for for $\mathbb{F}=\mathbb{R}$ and $\mathbb{C}$.

\begin{lemma}[Szemer\'edi-Trotter]\label{thm:ST}
For $\mathbb{F} = \mathbb{R}$ or $\mathbb{F} = \mathbb{C}$,
\[g_{k} \leq O(n^2 / k^3 + n/k).\]
\end{lemma}

\begin{proof}[Proof of Theorem \ref{thm:R2}]
Note that, since $\sum_k g_{=k} \leq n$,
\[\sum_{k \leq n^{1/3}} k^2 n \leq n^{5/3}.\]
Combining this with Lemmas \ref{thm:k-rich} and \ref{thm:ST},
\begin{align*}
|\Pi_{\alpha,\beta}| &\leq n^{5/3} + O \left( \sum_{n^{1/3} \leq k \leq s} (n^2 / k^2 + n) \right) \\
&\leq n^{5/3} + O(n^{5/3} + ns).
\end{align*}

\end{proof}

\subsubsection{Theorem \ref{thm:Fp2}}

In the case when $\mathbb{F}$ is a finite field with prime order, we can use an incidence theorem that was first proved by Bourgain, Katz, and Tao \cite{bourgain2004sum}  to obtain a slight improvement over Theorem \ref{thm:generalPlane} when $|P|$ is not too large.

\begin{lemma}\label{thm:BKT}
Let $P$ be a set of points and $L$ be a set of lines in $\mathbb{F}_p^2$ for prime $p$, with $|P|,|L| \leq N < p$.
Then there is a constant $\epsilon>0$ such that $I(P,L) \leq O(N^{3/2 - \epsilon})$.
\end{lemma}

The value of $\epsilon$ in Lemma \ref{thm:BKT} was improved by Jones \cite{jones2012further}, and recent improvements to the sum-product theorem in $\mathbb{F}_p^2$ by Roche-Newton, Rudnev, and Shkredov \cite{roche2014new} give further improvements to $\epsilon$.

\begin{proof}[Proof of theorem \ref{thm:Fp2}]
Lemma \ref{thm:BKT} implies that
\[k g_{k} \leq n^{3/2 - \epsilon}.\]
Hence, by Lemma \ref{thm:k-rich},
\[|\Pi_{\alpha, \beta}| \leq 2\sum_{k \leq s} O(n^{3/2 - \epsilon}) \leq O(s n^{3/2 - \epsilon}). \]

\end{proof}

\subsection{Higher dimensions}\label{sec:ddim}

In this section, we prove Theorems \ref{thm:generalHighDim}, \ref{thm:F3Rudnev}, and \ref{thm:Rd}.
These theorems are for sets of points in some higher dimensional space, and are united by the hypotheses that no $s$ points are contained in a hyperplane and no $t$ points are contained in a $(d-2)$-plane.

Given a set $H$ of hyperplanes and a set $P$ of points, we define the incidence graph $G(H,P)$ to be the bipartite graph with left vertices corresponding to the hyperplanes of $H$, right vertices corresponding to the points of $P$, and $(h,p) \in E(G)$ if and only if $p \in h$.
We denote the complete bipartite graph with $s$ left and $t$ right vertices as $K_{s,t}$.

\begin{lemma}\label{thm:noK2t}
Suppose $H$ is a set of hyperplanes in $\mathbb{F}^d$, and $P$ is a set of points, such that no $t$ points are contained in any single $(d-2)$-plane.
Then the incidence graph $G(H,P)$ does not include $K_{2,t}$ as a subgraph.
\end{lemma}

\begin{proof}
Since the hyperplanes of $H$ are distinct, the intersection of any two hyperplanes of $H$ is a $(d-2)$-plane, which does not contain $t$ points by hypothesis.
\end{proof}

\subsubsection{Theorem \ref{thm:generalHighDim}}

The classic bound of K\H{o}v\'ari, S\'os, and Tur\'an \cite{kovari1954problem} gives an upper bound on the number of edges in a $K_{2,t}$-free graph.

\begin{lemma}[K\H{o}v\'ari-S\'os-Tur\'an]\label{thm:KST}
Let $G$ be a $K_{2,t}$-free bipartite graph with $m$ left vertices and $n$ right vertices.
Then the number of edges of $G$ is at most $O(t^{1/2}mn^{1/2} + n)$.
\end{lemma}

From this, we can derive an upper bound on $|\Pi_{\alpha,\beta}|$ for a set of points in $\mathbb{F}^d$, for an arbitrary field $\mathbb{F}$.

\begin{proof}[Proof of Theorem \ref{thm:generalHighDim}]
Since no $t$ points of $P$ lie in any $(d-2)$-plane, Lemma \ref{thm:noK2t} implies that the incidence graph of $P$ with an arbitrary set of hyperplanes is $K_{2,t}$-free.
Hence, by Lemma \ref{thm:KST},
\[kg_k \leq O(t^{1/2}g_kn^{1/2} + n).\]
Hence, either $k \leq O(t^{1/2}n^{1/2})$, or $g_k \leq n/k$.
For $k \leq O(t^{1/2}n^{1/2})$, since $\sum_{k} g_{=k} \leq n$, we have
\[\sum_{k \leq t^{1/2}n^{1/2}} k^2 g_{=k} \leq tn^2.\]
When $s \geq O(t^{1/2}n^{1/2})$, by Lemma \ref{thm:k-rich} we have
\[|\Pi_{\alpha,\beta}| \leq O(tn^2) + \sum_{k \leq s} O(n) \leq O(tn^2),\]
since $s \leq n$.
The term $2s^2n$ in the bound comes from Corollary \ref{thm:s2n}.
\end{proof}

\subsubsection{Theorem \ref{thm:F3Rudnev}}

Rudnev gave an improvement to Lemma \ref{thm:KST} for incidences between points and planes in $\mathbb{R}^3$, under the condition that, if $\mathbb{F}$ has positive characteristic $p$, then the number of planes is $O(p^2)$.

\begin{lemma}\label{thm:rudnev}
Let $P,H$ be sets of points and planes, of cardinalities respectively $n$ and $m$, in $\mathbb{F}^3$.
Suppose $n \geq m$, and if $\mathbb{F}$ has positive characteristic $p$, then $p \neq 2$ and $m = O(p^2)$. 
Let $t$ be the maximum number of collinear planes.
Then,
\[I(P,H) \leq O(n \sqrt{m} + tn).\]
\end{lemma}

\begin{proof}[Proof of Theorem \ref{thm:F3Rudnev}]
Since the $g_k \leq n$,
\[\sum_{k < n^{1/2}} k^2 g_{=k} \leq n^2.\]

Lemma \ref{thm:rudnev} gives
\[kg_k \leq O(n \sqrt{g_k} + tn).\]

For $s \geq n^{1/2}$,
\begin{align*}
\sum_{k < s} k^2 (g_k - g_{k+1}) &\leq n^2 + \sum_{n^{1/2} < k < s} O(n^2/k + tn), \\
&\leq O(n^2 \log(sn^{-1/2}) + stn).
\end{align*}

Combined with Lemma \ref{thm:k-rich}, this finishes the proof.
\end{proof}

\subsubsection{Theorem \ref{thm:Rd}}

For $\mathbb{F} = \mathbb{R}$, we can use the following special case of an incidence bound of Lund, Sheffer, and de Zeeuw \cite{lund2014bisector}, based on the work of Fox, Pach, Sheffer, Suk, and Zahl \cite{fox2014semi}.
For the special case $d=3$, an earlier result of Basit and Sheffer would be sufficient for our purposes \cite{basit2014incidences}.

\begin{lemma}\label{thm:LSZ}
Let $H$ be a set of $m$ hyperplanes, and $P$ a set of $n$ points, both in $\mathbb{R}^d$, such that the incidence graph $G(H,P)$ is $K_{2,t}$-free.
Then, for any $\epsilon > 0$,
\[I(H,P) \leq O_{d,\epsilon}\left (m^{2(d-1)/(2d-1) + \epsilon} n^{d/(2d-1)} t^{(d-1)/(2d-1)} + tm + n \right).\]
\end{lemma}

\begin{proof}[Proof of Theorem \ref{thm:Rd}]

By Lemma \ref{thm:LSZ},
\[kg_k \leq O_{d,\epsilon} \left (g_k^{2(d-1)/(2d-1) + \epsilon} n^{d/(2d-1)} t^{(d-1)/(2d-1)} + tg_k + n \right).\]

For $k \leq O(t)$, this is trivial.

Let $\epsilon' = (2d - 1) \epsilon / (1 - (2d-1)\epsilon)$ \textemdash \, note that $\epsilon'$ is a function of $\epsilon$ that increases monotonically for $\epsilon > 0$, and that $\lim_{\epsilon \rightarrow 0} \epsilon ' = 0$.

Let 
\[r = \max(t, (nt)^{(1+2\epsilon')(d-1)/(2d-1)}).\]
Since $g_k \leq n$,
\[\sum_{k \leq r} g_{=k}k^2 \leq nr^2.\]

For $k \geq \Omega(r)$, we have
\[g_k \leq O_{d,\epsilon} \left((n^{d}t^{d-1}/k^{2d-1})^{1+\epsilon'} + n/k \right). \]

Hence,
\begin{align*}
\sum_{k < s} k^2(g_k - g_{k+1}) &\leq O_{d,\epsilon} \left( nr^2 + \sum_{r < k < s} \left( (n^d t^{d-1}k^{-2d + 2})^{1+\epsilon'} + n \right) \right), \\
& \leq O_{d,\epsilon} \left( nr^2 + (n^d t^{d-1} r^{-2d + 3})^{1 + \epsilon'} + ns \right), \\
& \leq O_{d,\epsilon} \left( nt^2 + (n^{(4d-3)/(2d-1)}t^{(2d-2)/(2d-1)})^{1+2\epsilon'} + ns\right) .
\end{align*}

Applying Lemma \ref{thm:k-rich} completes the proof.

\end{proof}

\bibliographystyle{plain}
\bibliography{dotProducts}

\end{document}